\documentclass[11pt]{amsart}
\usepackage{amssymb,latexsym,epsf}
\usepackage{hyperref}
\usepackage{url}
\usepackage{longtable}
\usepackage{mathrsfs}
\usepackage{multirow}
\usepackage{bigstrut}
\usepackage{amssymb}
\usepackage{graphicx,bm,epsfig,psfrag,subfigure}
\usepackage{amsfonts,amsmath,amssymb}

\usepackage{amsmath}
\usepackage{amsthm}
\usepackage{mathrsfs}

\newcommand{\C}{\mathbb{C}}

\newcommand{\Q}{\mathbb{Q}}
\newcommand{\R}{\mathbb{R}}
\newcommand{\Z}{\mathbb{Z}}

\newcommand{\PP}{\mathbb{P}}

\def\be{\begin{equation}}
\def\ee{\end{equation}}

\numberwithin{equation}{section}

\newtheorem{theorem}{Theorem}[section]
\newtheorem{prop}[theorem]{Proposition}
\newtheorem{lemma}[theorem]{Lemma}
\newtheorem{remark}[theorem]{Remark}

\newtheorem{definition}[theorem]{Definition}
\newtheorem{cor}[theorem]{Corollary}

\newtheorem{conj}[theorem]{Conjecture}

\begin{document}

\title{Luttinger surgery and Kodaira dimension}

\author{Chung-I Ho \& Tian-Jun Li}
\address{School  of Mathematics\\  University of Minnesota\\ Minneapolis, MN 55455}
\email{tjli@math.umn.edu}
\address{School  of Mathematics\\  University of Minnesota\\ Minneapolis, MN 55455}
\email{hoxxx090@math.umn.edu}

\begin{abstract}
In this note we show that the Lagrangian Luttinger
surgery preserves  the symplectic Kodaira dimension. Some constraints on Lagrangian
tori in symplectic four manifolds with non-positive Kodaira dimension are also derived.
\end{abstract}
\maketitle

\section{Introduction}
Let $(X, \omega)$ be a symplectic 4-manifold  with a Lagrangian torus $L$.
It was discovered  by Luttinger in \cite{Lu} that there is a  family of  surgeries along $L$
that produce symplectic $4$-manifolds. This family is countable and indexed by the pairs $([\gamma], k)$, where $[\gamma]$ is an isotopy class of  simple closed curves  on $L$ and $k$ is an integer. 
When $X=\R^4$ and  $\omega$ is the standard symplectic form $\omega_0=dx_1\wedge dy_1+dx_2\wedge dy_2$, he
also applied Gromov's celebrated work in \cite{Gr} to show that, for any Lagrangian torus $L$, all the resulting symplectic  manifolds  are  symplectomorphic to $(\R^4, \omega_0)$. 
This does not occur in general; a Luttinger surgery often fails even to preserve homology. 
As a matter of fact, many new exotic small manifolds
are constructed via this surgery.
In this note, we observe that the Luttinger surgery preserves one basic invariant:
\begin{theorem}\label{kod}
The Luttinger surgery
preserves the symplectic Kodaira dimension.
\end{theorem}

The symplectic Kodaira dimension of a symplectic 4-manifold $(X,\omega)$ is
defined by the products $K_{\omega}^2$ and $K_{\omega}\cdot [\omega]$,
where $K_{\omega}$ is the symplectic canonical class;
if $(X,\omega)$ is minimal, then
$$\kappa(X, \omega)=\left\{\begin{array}{ll}
-\infty & K_{\omega}^2<0\ or\ K_{\omega}\cdot [\omega] <0\\
0 & K_{\omega}^2=0\ and\ K_{\omega}\cdot [\omega]=0 \\
1 & K_{\omega}^2=0\ and\ K_{\omega}\cdot [\omega]>0\\
2 & K_{\omega}^2>0\ and\ K_{\omega}\cdot [\omega] >0
                  \end{array}\right.$$
For a general symplectic  4-manifold, the Kodaira dimension  is
defined as the Kodaira dimension of any of its minimal models. According to \cite{L},
$\kappa(X, \omega)$ is  independent of the choice
of symplectic form $\omega$ and hence is denoted by $\kappa(X)$.

Theorem \ref{kod} is related to a question of Auroux in \cite{A} (see Remarks \ref{omega} and \ref{44}).
Furthermore, together with the elementary analysis of the homology change,
the invariance of $\kappa$ implies that

\begin{theorem}\label{int2}
Let $(X, \omega)$ be a symplectic 4-manifold with $\kappa(X)=-\infty$ and $(\tilde{X}, \tilde{\omega})$ be constructed from $(X, \omega)$ via a Luttinger surgery. 
Then $(\tilde{X},\tilde{\omega})$ is symplectomorphic to $(X,\omega)$.
\end{theorem}

For minimal symplectic manifolds of Kodaira dimension zero, i.e., symplectic Calabi-Yau surfaces,
 we conclude that the Luttinger surgery is a symplectic CY surgery.
 Moreover, together with the homology classification of such manifolds in \cite{L4}, we have

\begin{theorem}\label{int3}
Suppose $(X, \omega)$ is a symplectic 4-manifold with $\kappa(X)=0$ and $\chi (X)>0$. 
If  $(\tilde{X}, \tilde{\omega})$ is constructed from $(X, \omega)$ under a Luttinger surgery, then $X$ and $\tilde{X}$ have the same integral homology type. 
\end{theorem}

In fact, we conjecture that $\tilde{X}$ and $X$ in Theorem \ref{int3} are diffeomorphic to each other.
For symplectic CY surfaces with $\chi =0$, the only known examples are torus bundles over torus.
We conjecture that they all can be obtained from $T^4$ via Lutttinger surgeries  (Conjecture \ref{torus}).

Theorems \ref{int2} and \ref{int3} provide topological constraints, phrased in terms of topological preferred framings (see Definition \ref{tprefer}),  on the existence of exotic Lagrangian tori in such manifolds.
\begin{theorem}\label{int4}
Let  L be a Lagrangian torus in $(X, \omega)$. 
If $\kappa (X)=-\infty$, or $L$ is null-homologous, $\kappa(X)=0$ and $\chi(X)>0$, then the Lagrangian framing of $L$ is topological preferred.
In particular, the invariant $\lambda (L)$ in \cite{FS} vanishes whenever it is defined.
\end{theorem}
 
The organization of this paper is as follows. In section 2, the construction of the Luttinger surgery is reviewed.
We also discuss the Lagrangian fibrations as the first application of this surgery.
In section 3, we establish the invariance of the symplectic Kodaira dimension, which is the main result of this note.
In section 4, we prove Theorems \ref{int2} and \ref{int3}.
In section 5, we apply these two theorems to derive constraints on framings of Lagrangian tori in symplectic 4-manifolds with non-positive Kodaira dimension. 

We are grateful to the Referee for many useful suggestions which improve the exposition.
The first author would like to thank the following scholars for their insightful discussions and suggestions: Anar Akhmedov, Inanc Baykur, Joel Gomez, Robert Gompf, Conan Leung, Weiwei Wu and Weiyi Zhang. The second author is grateful to the support of NSF.



\section{Luttinger surgery}

In this section, we describe the Luttinger surgery following \cite{ADK}.  
Applications to Lagrangian fibrations are also discussed.
We assume all manifolds are oriented.

\subsection{Construction}\label{construction}
Topologically, Luttinger surgery is a framed torus surgery. 
We start with a general description of framed torus surgeries. 
Let $X$ be a smooth 4-manifold and $L\subset X$ an embedded 2-torus with trivial normal bundle.
Then let $U$ be a tubular neighborhood of $L$.
If we assume $Y=X-U$ is the complement of $U$, $Z=\partial Y=\partial \overline{U}$ and $g:Z\rightarrow Z$ is a diffeomorphism,
a new manifold $\tilde{X}$ can be constructed by cutting $U$ out of $X$ and gluing it back to $Y$ along $Z$ via $g$:
\begin{align}\label{surgery}
\tilde{X}=Y\cup_gU.
\end{align}
Such surgery is called a $torus$ $surgery$.

It is often more explicit to describe this process via a framing of $L$.

\begin{definition}
Let $X$, $L$, $U$ and $Z$ be given as above.
A diffeomorphism $\varphi :U\rightarrow T^2\times D^2$ is called a framing of $L$ if $\varphi^{-1}(T^2\times 0)=L$.
Let $\pi_1: T^2\times D^2\rightarrow T^2$ be the projection.
For any $\gamma\subset L$ and $z\in \partial D^2$, the lift $\gamma_{\varphi}=\varphi^{-1}(\pi_1(\gamma )\times z)$ of $\gamma$ in $Z$ is called a longitudinal curve of $\varphi$.
Let
$$\partial \varphi :Z\rightarrow \partial (\overline{T^2\times D^2})\cong T^2\times S^1$$
be the induced map.
Two framings $\varphi _1, \varphi_2: U\rightarrow T^2\times D^2$ are smoothly isotopic to each other if
the map
$$\partial \varphi_2\circ (\partial \varphi_1)^{-1}:T^2\times S^1\rightarrow T^2\times S^1$$
is homotopic to the identity map.
\end{definition}

$\partial \varphi$ induces a $S^1$-bundle structure on $Z$.
A positive oriented fiber $\mu$ of $Z$ is called a $meridian$ of $L$.
For $\tilde{X}$ in \eqref{surgery}, we will use $\tilde{L}$ to denote the torus $L\subset U\subset \tilde{X}$.
Notice that $\tilde{L}$ also inherits a framing $\tilde{\varphi}$ and its meridian $\tilde{\mu}\subset Z$ satisfies
$$[\tilde{\mu}]=p[\mu]+k[\gamma_\varphi],$$
in $H_1(Z;\Z)$. 
Here $\gamma_\varphi$ is a longitudinal curve of $\varphi$ and $p,k$ are coprime integers.
The diffeomorphism type of $\tilde{X}$ only depends on the class $[\tilde{\mu}]$.
It is called a \textit{generalized logarithmic transform} of $X$ along $(L,\varphi ,\gamma)$ with multiplicity $p$ and auxiliary multiplicity $k$, or of type $(p, k)$ (see \cite{GoS}), and denoted as $X_{(L, \varphi, \gamma, p, k)} $.
For brevity, we will call it a $(p,k)$-surgery.

If $X$ is a symplectic 4-manifold, Weinstein's
theorem states that there is a canonical framing for any Lagrangian torus of $X$.
\begin{definition}
Let $X$ be a symplectic 4-manifoldand and $L$ a Lagrangian torus of $X$.
A framing $\varphi$ of $L$ is called a Lagrangian framing if $\varphi^{-1}(T^2\times z)$ is a Lagrangian submanifold of $X$ for any $z\in D^2$.
\end{definition}

Topologically, a Luttinger surgery is a $(1, k)$-surgery with respect to a Lagrangian framing.
In order to deal with the sympelctic structure, it is more convenient to use a square neighborhood  rather than the disk neighborhood of $L$ as above. 

Express the cotangent bundle $T^*T^2$  as
$$\{(x_1, x_2,y_1,y_2)\in\R^4\}/(x_1=x_1+1, x_2=x_2+1)$$
equipped with the canonical 2-form
$$\omega_0=dx_1\wedge dy_1+dx_2\wedge dy_2$$
and let
$$U_r= \{(x_1,x_2, y_1, y_2)\in T^*T^2|-r<y_1<r, -r<y_2<r\},$$
There exists a tubular neighborhood $U$ of $L$ and a symplectomorphism
$\varphi : (U,\omega )\rightarrow (U_r,\omega_0)$
for small $r$ which satisfies
$$  \varphi (L)=T^2\times (0, 0).$$
In addition, given a simple closed curve $\gamma$ on $L$,
we can choose the coordinates $x_1,x_2$ of $T^2$ such that
$$\varphi(\gamma) =\{(x_1,0,0,0)\mid x_1\in \R/\Z\}.$$

Let $A_{s,t}=U_s-\overline{U_t}\ (s>t)$
be an annular region and $f:(-r,r)\rightarrow [0,1]$ be a smooth increasing function
such that $f (t)=0$ for $t\leq -\frac{r}{3}$ and $f(t)=1$
for $t\geq \frac{r}{3}$.
For any integer $k$, we can define a diffeomorphism $h_k$:
$$\begin{array}{ccl}
 A_{r,\frac{r}{2}} & \rightarrow & A_{r,\frac{r}{2}}\\
(x_1,x_2,y_1,y_2) & \mapsto &
\left\{\begin{array}{ll} (x_1+kf (y_1),x_2,y_1,y_2) & y_2\geq\frac{r}{2}\\
(x_1,x_2,y_1,y_2) & \hbox{otherwise}
\end{array}\right.
\end{array}$$

Observe that
\begin{equation}\label{hk}
h_k^*(\omega_0)=\omega_0,
\end{equation}
which follows from the relation
$$(dx_1+kf'(y_1)dy_1)\wedge dy_1+dx_2\wedge dy_2=\omega_0$$
for $y_2\geq\frac{r}{2}$.

Let $X_L=X-\overline{\varphi^{-1}(U_{\frac{r}{2}})}$
and define $g_k=\varphi^{-1}\circ h_k\circ\varphi$ via the following diagram
$$\begin{array}{rcccl}
X_L\supset & U-\overline{\varphi^{-1}(U_{\frac{r}{2}})} & \stackrel{g_k}{\rightarrow} & U-\overline{\varphi^{-1}(U_{\frac{r}{2}})} & \subset U\\
&\downarrow\varphi &  & \varphi \downarrow\\
& A_{r,\frac{r}{2}} & \stackrel{h_k}{\rightarrow} & A_{r,\frac{r}{2}}\\
\end{array},$$
then we can construct a new smooth manifold
$$\tilde{X}:=X_L\cup_{g_k} U.$$
Notice that, by \eqref{hk}, we have $g_k^*(\omega)=\omega$. 
Thus $\tilde{X}$ carries a symplectic form $\tilde{\omega}$ induced by $\omega$.
This process is called a \textit{Luttinger surgery} (along the Lagrangian torus $L$).

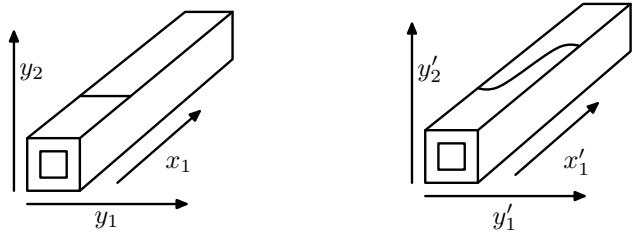
\begin{center}
\epsfbox{lk.1}
\hspace{2cm}
\epsfbox{lk2.1}
\end{center}

We know that

 \begin{lemma}\label{framing}\cite{FS}
Any two Lagrangian framings of a Lagrangian torus are smoothly isotopic to each other.
\end{lemma}

Hence the symplectomorphism type of $(\tilde X, \tilde \omega)$ only depends on the Lagrangian
isotopy class of $L$, the  isotopy class of $\gamma$ in $L$, and the integer $k$.
Therefore, $\tilde X$ is also denoted as $X(L,\gamma, k)$.

It is worth mentioning that a Luttinger surgery can be reversed.
Let $\tilde{L}, \tilde{\gamma}$ be the subsets $\varphi^{-1}(T^2\times (0, 0))$ and $\varphi^{-1}(\R /\Z \times (0, 0, 0))$ of $\tilde{X}$.
We can apply
the Luttinger surgery to $X(L,\gamma, k), \tilde L, \tilde \gamma$ with coefficient $-k$ to recover $X$.


\subsection{Lagrangian fibrations}
One natural source of Lagrangian tori is smooth fibers of Lagrangian fibrations.
\begin{definition}
Let $(X,\omega)$ be a symplectic  4-manifold, and let $B$ be a 2-manifold (with boundary or vertices).
A smooth map $\pi:X \rightarrow B $ is called a Lagrangian fibration if there exists an open dense subset $B_0\subset B$ such that $\pi^{-1}(b)$ is a compact Lagrangian submanifold of $X$ for any $b\in B_0$.
$X$ is called Lagrangian fibered if such a structure exists.
\end{definition}

It is easy to see that any smooth fiber of a Lagrangian fibration must be a torus.
Moreover, we have

\begin{lemma}\label{fiber}
A Luttinger surgery along a Lagrangian fiber preserves the Lagrangian fibration structure.
\end{lemma}

\begin{proof}
Let $\pi :X\rightarrow B$ be a Lagrangian fibration and $L=\pi^{-1}(b)\subset X$ a generic fiber.
Using notations from section 2.1, it is shown in \cite{M} that there is a
neighborhood $B_r$ of $b$ and $U=\pi^{-1}(B_r)$ with local charts $\varphi : (U,\omega )\rightarrow (U_r,\omega_0)$ and $\varphi_0:B_r\rightarrow D_r=(-r, r)\times (-r, r)$ such that
the diagram
$$\begin{array}{rcc}
 U & \stackrel{\varphi}{\longrightarrow} & U_r \\
\pi\downarrow &  &  \downarrow \pi_0\\
 B_r & \stackrel{\varphi_0}{\longrightarrow} & D_r\\
\end{array}$$
commutes. Here $\pi_0$ is the projection $(x_1, x_2, y_1, y_2)\mapsto (y_1, y_2)$.

If $\tilde{X}=X_L\cup_{g_k} U$ is obtained by performing Luttinger surgery along $L$ (indexed by $\gamma\subset L$ and $k\in\Z$), we can define a map $\tilde{\pi}:\tilde{X}\rightarrow B$ as
$\tilde{\pi}(\tilde{x})=\pi(x)$.
Since $\pi_0\circ h_k=\pi_0$, we have
$$\pi\circ g_k= \pi\circ \varphi^{-1}\circ h_k\circ\varphi= \varphi_0^{-1}\circ \pi_0\circ h_k\circ\varphi=
\varphi_0^{-1}\circ \pi_0\circ\varphi= \pi$$
So $\tilde{\pi}$ is well-defined. It is clear that $\tilde{\pi}$ is Lagrangian and
$\tilde{X}$ also possesses a Lagrangian fibration structure.
\end{proof}

Lagrangian fibrations appear widely in toric geometry, integral systems and mirror symmetry.
We will discuss almost toric fibration
introduced by Symington in some detail.

\begin{definition}
An almost toric fibration of a symplectic 4-manifold $(X, \omega)$ is a Lagrangian
fibration $\pi :X\rightarrow B$ with the following properties:
for any critical point $x$ of $\pi$, there exists a local
coordinate $(x_1,x_2, y_1, y_2)$ near $x$ such that
$x=(0, 0, 0, 0)$, $\omega =dx_1\wedge dy_1+dx_2\wedge dy_2$,
and $\pi$ has one of the forms
$$(x_1, x_2, y_1, y_2) \rightarrow \left\{\begin{array}{l} (x_1^2+y_1^2,x_2^2+y_2^2)\\
(x_1^2+y_1^2, x_2)\\
(x_1^2-y_1^2, x_2)
\end{array}\right.$$
An almost toric 4-manifold is a symplectic 4-manifold equipped with an almost toric fibration.
\end{definition}

The base $B$ of an almost toric fibration has an affine structure with boundary and vertices.
Moreover, these three types of critical points project to vertices, edges and interior of $B$ respectively.
Almost toric fibrations are classified by Leung and Symington:

\begin{theorem}\label{almost}\cite{LS}
Let $(X,\omega)$ be a closed almost toric 4-manifold.
There are seven types of almost toric fibrations according to the homeomorphism type of the base $B$.
\begin{enumerate}
\item $\C\PP^2 \sharp n\overline{\C\PP^2}$ or $S^2\times S^2$, $B$ is (homeomorphic to) a disk;
\item $(S^2\times T^2)\sharp n\overline{\C\PP^2}$ or $(S^2\tilde{\times} T^2) \sharp n\overline{\C\PP^2}$,
$B$ is a cylinder;
\item $(S^2\times T^2)\sharp n\overline{\C\PP^2}$ or $(S^2\tilde{\times} T^2) \sharp n\overline{\C\PP^2}$,
$B$ is a M\"{o}bius band;
\item the $K3$ surface, $B$ is a sphere;
\item the Enriques surface, $B$ is $\R\PP^2$;
\item a torus bundle over torus with monodromy
$$\left\{I,\left(\begin{array}{cc}
1 & m\\
0 & 1 \end{array}\right)\right\}, m\in\Z$$
$B$ is a torus;
\item a torus bundle over the Klein bottle with monodromy
$$\left\{\left(\begin{array}{cc}
1 & 0\\
0 & -1 \end{array}\right),\left(\begin{array}{cc}
1 & m\\
0 & 1 \end{array}\right)\right\}, m\in\Z$$
$B$ is a Klein bottle.
\end{enumerate}
\end{theorem}

An immediate consequence of this classification is the calculation of the symplectic Kodaira dimension.
\begin{prop}\label{al}
If $(X, \omega)\rightarrow B$ is an almost toric fibration, then $\kappa(X)\leq 0$.
Moreover, $\kappa(X)=0$ if and only if the base $B$ is closed.
\end{prop}

The effect of Luttinger surgeries on almost toric fibrations is also easy to describe.

\begin{prop}\label{K3}
Suppose $(X,\omega)\rightarrow B$ is an almost toric fibration and $(\tilde{X},\tilde{\omega})$ is obtained from $(X, \omega)$ by performing a Luttinger surgery along a smooth fiber $L$, then $(\tilde{X}, \tilde \omega)$ retains an almost toric fibration structure with the same base.
Moreover, $\tilde{X}$ is diffeomorphic to $X$ if $\chi(B)>0$.
\end{prop}
\begin{proof}
The first statement is given by Lemma \ref{fiber}.
If $\chi(B)>0$, $X$ and $\tilde{X}$ are in one of the types (1)-(5) in Theorem \ref{almost}.
In each of them, the list of manifolds are distinguished by the type of intersection forms and Euler numbers.
So the second result for types (1)-(3) follows from Proposition \ref{compare} and the fact that homology classes of Lagrangian tori in manifolds with $b^+=1$ are torsion.
It is clear for (4) and (5) from the classification.
\end{proof}

Propositions \ref{al} and \ref{K3} provide examples of Luttinger surgeries preserving the symplectic Kodaira dimension.
In the next section, we will show that it is true for any Luttinger surgery.



\section{Preservation of  Kodaira dimension}
In this section, we prove Theorem \ref{kod}.
To proceed, we must first prove the invariance of minimality under Luttinger surgery.


\subsection{Minimality} A symplectic (smooth) $-1$ class is a degree 2 homology class represented by an embedded symplectic (smooth) sphere
with self-intersection $-1$. A symplectic 4-manifold is called $symplectically$ $(smoothly)$
$minimal$ if it does not have any  symplectic (smooth) $-1$ class.
The symplectic minimality is actually equivalent to
smooth minimality.

\begin{prop} \label{minimality}
The Luttinger surgery preserves the minimality.
\end{prop}

\begin{proof}
Since a Luttinger surgery can be reversed and the reverse operation is also
a Luttinger surgery,
it suffices to show that, if we start with  a
non-minimal symplectic 4-manifold, then after a Luttinger surgery,
the resulting symplectic manifold is still  non-minimal. But this is
a direct consequence of  the  following fact in \cite{W}:

\begin{theorem} \label{disjoint} Given a Lagrangian torus
$L$ and a symplectic $-1$ class, there is an embedded symplectic
$-1$ sphere in that class which is disjoint from $L$.
\end{theorem}
\end{proof}


\subsection{Kodaira dimension}
Now, we analyze the effect of Luttinger surgery on the symplectic canonical
class $K_{\omega}$ and the symplectic class $[\omega]$.
Recall that $X_L$ is an open submanifold
of both $X$ and $\tilde X$, and let $\nu :X_L \rightarrow X$
and $\tilde{\nu}:X_L\rightarrow \tilde{X}$ be the inclusions.

To prepare for the following lemma, we use the notations from section \ref{construction}.
For the sake of simplicity, we will identify any object in $X$ with their image of $\varphi$
and $(x_1, x_2, y_1, y_2)$, $(x_1', x_2', y_1', y_2')$ will denote
the coordinates of $A_{r,\frac{r}{2}}$ on $X_L$ and $U$ respectively.

\begin{lemma}\label{thom}
There exists a 2-dimensional submanifold $S\subset X_L$ such that
$\nu_*([S])=PD(K_{\omega})\in H_2(X)$ and $\tilde{\nu}_*([S])=PD(K_{\tilde{\omega}})\in
H_2(\tilde{X})$.
\end{lemma}
\begin{proof}

Let $J$ be a $\omega$-tamed almost complex structure in $X$ which
induces a complex structure on $T^*U$ as $$J(dx_1)=-dy_1,\ J(dx_2)=-dy_2$$
Assume $\rho :(-r, r)\rightarrow [0, 1]$ is a continuous increasing function satisfying
$$\rho (t)=\left\{\begin{array}{ll}
0 & t\leq 0\\
1 & t>\frac{r}{3}.\\
                  \end{array}\right.$$
Another almost complex structure $J'$ in $T^*U$ is defined as
$$J'(dx_1')=-kf'(y_1)\rho(y_2)dx_1'-(k^2f^2(y_1)\rho(y_2)+1)dy_1',\ J'(dx_2')=-dy_2'$$
It is easy to check that $J'$ is $\omega$-tamed and $(g_k)_*(J)=J'$ in $X_L\cap U$.

Let $\pi: \mathcal{L}\rightarrow X$ and $\tilde{\pi}:
\tilde{\mathcal{L}}\rightarrow \tilde{X}$ be the canonical bundles
of $X$ and $\tilde{X}$, respectively, and let $s:X\rightarrow \mathcal{L}$
and $\tilde{s}:\tilde{X}\rightarrow \tilde{\mathcal{L}}$ denote the corresponding
embeddings of zero sections. Since $\mathcal{L}$ is trivial
on $U$, we can find a global section
$\sigma$ of $\mathcal{L}$ and a Thom class $\Phi\in
H^2_{cv}(\mathcal{L})$ such that $\sigma= (dx_1+idy_1) \wedge
(dx_2+idy_2)$ in $X_L\cap U$ and $\Phi=0$ in $s(U)$. Another nonzero
$(2, 0)$-form in $U$ is constructed as
$$\sigma'=(dx_1'+iJ'(dx_1'))\wedge  (dx_2'+iJ'(dx_2'))$$
In $X_L\cap U$, we have
\begin{eqnarray*}
& &g_k^*(\sigma')\\
&=& g_k^*((dx_1'+iJ'(dx_1'))\wedge  (dx_2'+iJ'(dx_2')))\\
&=& g_k^*((dx_1'+i(-kf'(y_1)\rho(y_2)dx_1'+(k^2f'^2(y_1)\rho(y_2)+1)dy_1')\wedge (dx_2'+idy_2'))\\
&=& (dx_1+kf'(y_1)dy_1+i(-kf'(y_1)dx_1+dy_1)) \wedge (dx_2+idy_2)\\
&=& (1-ikf'(y_1))(dx_1+idy_1) \wedge (dx_2+idy_2)\\
&=& (1-ikf'(y_1))\sigma
\end{eqnarray*}

$\sigma$ and $\sigma'$ give two local trivializations of
$\tilde{\pi}^{-1}(X_L\cap U)$ with transition function
$\theta=1-ikf'(y_1)$. Since $-\frac{\pi}{2}<$ arg$(\theta)<\frac{\pi}{2}$,
we can normalize the frame of $\tilde{\pi}^{-1}(U)$ such that $\theta=1$.
Hence $\Phi\mid_{\pi^{-1}(X_L)}$ can be
extended to $\tilde{\mathcal{L}}$ via constant function and form a Thom
class $\tilde{\Phi}$ satisfying

\begin{enumerate}
\item $\tilde{\Phi}=\Phi$ in $\mathcal{L}\mid_{X_L}\cong \tilde{\mathcal{L}}\mid_{X_L}(=\pi^{-1}(X_L))$.
\item $\tilde{\Phi}$ is independent of the coordinates $(x_1', x_2', y_1', y_2')$ in $\tilde{\pi}^{-1}(U)$. In particular, $\tilde{\Phi}=0$ in $\tilde{s}(U)$.
\end{enumerate}

It is clear that these 2-forms $e=s^*(\Phi)$ and $\tilde{e}=\tilde{s}^*(\tilde{\Phi})$
are equivalent in $X_L$ and vanish in $U\subset X$ and $U\subset \tilde{X}$ respectively.
Using these representations, we can find a 2-submanifold
$S\subset$ supp$(e)\subset X_L$ which is Poincar\'e dual to $K_{\omega}$ in $X$, and
dual to $K_{\tilde{\omega}}$ in $\tilde{X}$.
\end{proof}

The main theorem can be proved now.

\begin{proof}[Proof of Theorem \ref{kod}]
Suppose $\tilde{X}$ is obtained from $X$ by applying a Luttinger
surgery along $L$.
Let us first consider the case in which $X$ is minimal.  By Proposition
\ref{minimality}, $\tilde {X}$ is also minimal. Let $K_\omega$ and
$K_{\tilde{\omega}} $ denote the canonical classes of $X$ and $\tilde{X}$
respectively. 
By Lemma \ref{thom}, there exists a submanifold
$S\subset X_L$ such that $\nu_*([S])=PD(K_{\omega})$ and
$\tilde{\nu}_*([S])=PD(  K_{\tilde{\omega}})$. We also know that
$\omega=\tilde{\omega}$ in $X_L$. So
$$K_{\omega}^2=\int_S K_{\omega}=\int_S K_{\tilde{\omega}} = K_{\tilde{\omega}}^2$$
$$K_\omega\cdot [\omega] =\int_S\omega=\int_S\tilde{\omega}=K_{\tilde{\omega}}\cdot [\tilde{\omega}]$$
Thus  the Kodaira dimensions of $X$ and $\tilde{X}$ coincide.

If $X$ is not minimal, we can blow down $X$ along symplectic $-1$ spheres disjoint from $L$ to a minimal model.
These spheres are contained in $X_L$ and the same procedure can be applied to $\tilde{X}$, so we can argue as above.
\end{proof}

Theorem \ref{kod} can be used to distinguish non-diffeomorphic manifolds.
In \cite{ABBKP, ABP, AP1, AP2, BK, FPS}, several symplectic manifolds homeomorphic but not diffeomorphic to non-minimal rational surfaces are constructed.
With $\kappa=2$ for the building blocks, it also easily follows from Theorem \ref{kod} that they are exotic.

\begin{remark}\label{omega}
\begin{enumerate}
\item
The main theorem is proved based on the invariance of $K_\omega\cdot [\omega]$ and $K_\omega^2$.
Actually, the class $[\omega]^2$ is also preserved since the volume is invariant under a Luttinger surgery.
Theorem \ref{kod} is expected, in light of Auroux's Question 2.6 in \cite{A}:\\
\\
\textit{Let $X_1, X_2$ be two integral compact symplectic 4-manifolds with the smae $(K^2, \chi ,K\cdot [\omega], [\omega]^2)$. Is it always possible to obtain $X_2$ from $X_1$ by a sequence of Luttinger surgeries?}
\\
\item
It is well known that the Dolgachev surfaces $S(p, q)$ obtained by performing two logarithmic transforms with multiplicities $p>1, q>1$ to $\C\PP^2\sharp 9\overline{\C\PP^2}$ have $\kappa=1$.
So a generalized logarithmic transform may not preserve $\kappa$ (see a related discussion in \cite{BS}).
\end{enumerate}
\end{remark}





\section{Manifolds with non-positive $\kappa$}
In this section we apply Theorem \ref{kod} to study the effect of Luttinger surgeries on  symplectic 4-manifolds with $\kappa \leq 0$.

\subsection{Torus surgery and homology}\label{TF}
We start by analyzing how homology changes under a general torus surgery. 
Suppose $X$ is a smooth 4-manifold and $L\subset X$ is an embedded 2-torus with trivial normal bundle.
Moreover, $U, Y, Z, g, \tilde{X}$ are defined as in section \ref{construction}.

\subsubsection{}
To compare the homology of $X$ and $\tilde{X}$, we need to compare both of them with the homology of $Y$.
The inclusion $i:Z\rightarrow Y$ induces homomorphisms 

\begin{align}\label{i1}
i_k^\Z: H_k(Z; \Z)\rightarrow H_k(Y; \Z)
\end{align}

and 

\begin{align}\label{i2}
i_k^\Q: H_k(Z; \Q)\rightarrow H_k(Y; \Q)
\end{align}
in homology.
We often use $i_k$ to denote $i_k^\Q$ and $H_k(-)$ to denote $H_k(-,\Q)$.
We also use $r(A)$ to denote the dimension of any $\Q$-vector space $A$.

The following lemma is a well know fact, for which we offer a geometric argument.

\begin{lemma}\label{mul}
$[\mu]\in \ker i_1$ if and only if $[L]\ne 0$ in $H_2(X)$.
\end{lemma}
\begin{proof}
Suppose $i_1[\mu]=0$ in $H_1(Y)$, i.e. $l\, i_1^\Z[\mu]=0$ in $H_1(Y;\mathbb Z)$ for some positive integer $l$.
Thus $l$ copies of $\mu$ bounds an oriented  surface $A$ in $Y$. Extend $A$ by $l$
normal disks inside the tubular neighborhood to obtain a closed
oriented surface $A'$ intersecting with $L$ at $l$ points with the
same sign. This implies in particular that $[L]\ne 0$ in
$H_2(X)$.

Conversely, suppose $[L]\ne 0$ in $H_2(X)$, then there exists a closed
oriented surface $B$ in $X$ intersecting $L$ with nonzero
algebraic intersection numbers, say $l$. We may assume that the
intersection is transverse with $l+b$ positive intersection points
and $b$ negative intersection points. We can further assume that $B$
intersects the closure of $U$ at $l+2b$ normal disks, $l+b$ of
those having positive orientations, the remaining $b$ disks having
negative orientations. This implies that the complement of those
disks in $B$ is an oriented surface in $Y$, whose boundary is homologous
to $l\mu$, and thus $i_1[\mu]$ is zero.
\end{proof}

When we consider the integral homology,
Lemma \ref{mul} immediately implies

\begin{cor}\label{mulz}
If $[L]=0$ in $H_2(X;\Z)$, then $i_1^\Z[\mu]$ is a non-torsion class in $H_1(Y;\Z)$.
\end{cor}

Consider the Mayer-Vietoris sequence
\begin{equation}\label{mv}\cdots \stackrel{\partial_{k+1}}{\longrightarrow}H_k(Z)\stackrel{\rho_k}{\longrightarrow} H_k(Y)\oplus H_k(U)\stackrel{\nu_k}{\longrightarrow} H_k(X)\stackrel{\partial_k}{\longrightarrow}H_{k-1}(Z)\stackrel{\rho_{k-1}}{\longrightarrow} \cdots
\end{equation}
where $\rho_k=(i_k, j_k)$ and $\nu_k=\nu_k'\oplus (-\nu_k'')$ with $i_k, j_k, \nu_k', \nu_k''$ induced by inclusions.

\begin{lemma}\label{meridian}
\begin{enumerate}
\item $\partial_1=0$ and $\nu_1':H_1(Y)\rightarrow H_1(X)$ is surjective.

\item
$$ r(\hbox{Im} \rho_1)=
 \left\{\begin{array}{ll} 3 & \hbox{if } i_1[\mu]\neq 0\\
2 & \hbox{if } i_1[\mu] = 0 
\end{array}\right.$$
and $\rho_1$ is injective if and only if $[\mu]\notin \ker i_1$.

\item
\begin{equation}\label{xy}
H_1(X;\Z)\cong H_1(Y; \Z)/<i_1^\Z[\mu]>
\end{equation}
\item
If $[L]=0\in H_2(X)$, then $\nu_2':H_2(Y)\rightarrow H_2(X)$ is surjective.
\end{enumerate}
\end{lemma}

\begin{proof}
\begin{enumerate}

\item
It is clear that any class $a$ in $H_1(X;\Z)$ can be represented by a 1-cycle $C$ disjoint from $L$.
$C$ is also disjoint from $Z$ if the neighborhood $U$ is small enough.
So $\partial_1 a=[C\cap Z]=0$. $C\subset Y$ implies that $\nu_1'$ is surjective.

\item We know $\ker\rho_1=\ker i_1\cap \ker j_1\subset \ker j_1$.
Since $\ker j_1=<[\mu]>$,
$$\ker\rho_1=
 \left\{\begin{array}{ll} 0 & \hbox{if } i_1[\mu]\neq 0 \\
<[\mu]> & \hbox{if } i_1[\mu] = 0 
\end{array}\right.$$
and $\rho_1$ is injective if and only if $i_1[\mu]\neq 0$. The rank of Im$\rho$ is given from $r($Im$\rho_1)=r(H_1(Z))-r(\ker\rho_1)$.

\item
The sequence \eqref{mv} induces a short exact sequence
$$0\rightarrow H_1(Y;\Z)\oplus H_1(U;\Z)/\ker \nu_1\rightarrow H_1(X;\Z) \rightarrow
\hbox{Im}\partial_1 \rightarrow 0$$
$\partial_1=0$ implies that
$$H_1(Y;\Z)\oplus H_1(U;\Z)/\ker \nu_1 \cong H_1(X;\Z)$$
Because $\nu_1'$ is surjective, we also have
$$H_1(X;\Z)=H_1(Y;\Z)/\ker \nu_1'$$
If $\{[\mu], \gamma_1, \gamma_2\}$ is a basis of $H_1(Z;\Z)$,
then $\hbox{Im}\rho_1 =<([\mu] , 0),(\gamma_1,\gamma_1), (\gamma_2,\gamma_2)>$
and $\gamma_1, \gamma_2\neq 0\in H_1(U;\Z)$.
For $a\in H_1(Y;\Z)$, $a\in\ker \nu_1'$ if and only if $(a, 0)\in \ker\nu_1=$Im$\rho_1$, or $a=ki_1^\Z[\mu]$ for some $k\in\Z$.
So $\ker \nu_1'=<i_1^\Z[\mu]>$ and
 $$H_1(X;\Z)=H_1(Y;\Z)/<i_1^\Z[\mu]>$$

\item
$$\begin{array}{rclr}
[L]=0\in H_2(X) &\Leftrightarrow & [\mu]\neq 0\in H_1(Y) & (\hbox{by Lemma \ref{mul}} )\\
&\Leftrightarrow &\rho_1\hbox{ injective} & (\hbox{by part(2)})\\
&\Leftrightarrow &\partial_2=0 & (\hbox{exactness})\\
&\Leftrightarrow &\nu_2\hbox{ surjective} & (\hbox{exactness})
\end{array}$$
Since $[L]=0$ also implies $\nu_2''=0$, $\nu_2'$ has to be surjective.
\end{enumerate}
\end{proof}

All the results hold if we replace $X$, $L$, $\mu$ by $\tilde{X}, \tilde{L}$ and $\tilde{\mu}$.
Now, we are ready to  compare $X$ and $\tilde{X}$.

\subsubsection{Comparing $H_*(X)$ and $H_*(\tilde{X})$}

Lemma \ref{meridian}, applied to torus surgeries,  gives

\begin{prop}\label{compare}
If $\tilde{X}$ is obtained from $X$ via a torus surgery, then
\begin{enumerate}

\item $\chi(\tilde X)=\chi({X})$, $\sigma(\tilde X)=\sigma(X)$.

\item
$$b_1(\tilde{X})-b_1(X) =
 \left\{\begin{array}{ll} 0 & \hbox{if $i_1[\mu]=0=i_1[\tilde \mu]$ or $i_1[\mu]\neq 0 \neq i_1[\tilde \mu]$}\\
 -1 &\hbox{if $i_1[\mu]=0$ and $i_1[\tilde \mu]\ne 0$}\\
1  & \hbox{if $i_1[\mu]\ne 0$ and $i_1[\tilde \mu]= 0$}
\end{array}\right.
$$

\item $|b_1(\tilde{X})-b_1(X)|\leq 1$ and $|b_2(\tilde{X})-b_2(X)|\leq 2$.

\end{enumerate}
\end{prop}
\begin{proof}
\begin{enumerate}
\item Obvious.

\item
Since $\partial_1=0$, we can conclude that
$$b_1(X)=b_1(Y)+2-r(\hbox{Im}\rho_1) =
 \left\{\begin{array}{ll} b_1(Y)-1 & \hbox{if } i_1[\mu]\neq 0 \\
b_1(Y) & \hbox{if } i_1[\mu] = 0.
\end{array}\right.$$
The same is true for $b_1(\tilde{X})$ with $i_1[\mu]$ replaced by $i_1[\tilde{\mu}]$.
The proof is finished by comparing $b_1(X)$ and $b_1(\tilde{X})$.
\item The first inequality is given by part (2).
The second inequality follows from part (1) and the first inequality.
\end{enumerate}
\end{proof}

The next result concerns with the intersection forms. 

\begin{prop}\label{comparez}
Suppose $X$ and $\tilde{X}$ are defined as above.
If $[L]$ is a torsion class in $H_2(X;\Z)$ and the intersection form $Q(X)$ is odd, then $Q(\tilde{X})$ is odd as well.
In particular, if both $[L]$ and $[\tilde{L}]$ are torsion, then $\tilde{X}$ and $X$ have the same
intersection form.
\end{prop}
\begin{proof}
Since $Q(X)$ is odd, there exists a closed oriented  surface $S$ in $X$ such that $S\cdot S$ is odd.
By Lemma \ref{meridian}(4), $[S]\in \hbox{Im}\nu_2'$ and $S$ can be chosen such that $S\subset Y$.
Thus, $S$ is contained in $\tilde{X}$, and hence, $Q(\tilde{X})$ is also odd.
\end{proof}


\subsection{$\kappa=-\infty$}
By Proposition \ref{K3}, if a symplectic manifold $(X, \omega)$ with $\kappa (X)=-\infty$ has
an almost toric structure $\pi: X\rightarrow B$ and if we apply a Luttinger surgery along a smooth fiber of $\pi$, the new manifold $(\tilde{X}, \tilde{\omega})$ is diffeomorphic to $(X, \omega)$.
Such phenomenon is still true for any 4-manifold with $\kappa=-\infty$.
Moreover, we have the stronger Theorem \ref{int2}.

\begin{proof}[Proof of Theorem \ref{int2}]
Proposition \ref{minimality} allows us to reduce to the case where $(X, \omega)$ is
minimal.

We first show that $\tilde X$ is diffeomorphic to $X$.
Observe that the diffeomorphism types of minimal manifolds with
$\kappa=-\infty$ are distinguished by their
Euler numbers and intersection forms. Since such manifolds have $b^+=1$,
the homology classes of Lagrangian tori are torsion. Thus,
 both quantities are preserved by Proposition \ref{compare} and \ref{comparez}.

To show further that $(\tilde X, \tilde \omega)$ and $(X, \omega)$ are symplectomorphic to each other,
it is enough to show that $\omega$ is cohomologous to $\tilde{\omega}$
(\cite{Mc}).
If $X$ is diffeomorpic to $\C\PP^2$, the symplectic structure is determined by the volume $[\omega]^2$,
which is preserved by Remark \ref{omega}(1).

When $X$ is ruled, $H^2(X)$ is either generated by $K_\omega$ and the Poincar\'e dual to the homology class of a fiber $F=S^2$, or by $K_{\omega}$ and $[\omega]$.
Hence the class of $\omega$ is determined  by
$K_{\omega}\cdot [\omega]$, $[\omega]^2$ and $[\omega](F)$.
As mentioned above, the first two quantities are preserved.
By \cite{W} the fiber sphere can be chosen to be disjoint from $L$, so it follows
that the last quantity is also preserved.
\end{proof}

\subsection{Luttinger surgery as a symplectic CY surgery}

A symplectic CY surface is a symplectic 4-manifold with torsion canonical class, or equivalently, a minimal symplectic
4-manifold with $\kappa=0$.

By Theorem \ref{kod} and Proposition \ref{minimality}, we have

\begin{prop}\label{cy1}
A Luttinger surgery is a symplectic CY surgery in dimension four.
\end{prop}

There is a homological classification of symplectic CY surfaces in \cite{L4} and \cite{B}.

\begin{theorem}\label{cy2}
A symplectic CY surface is an integral homology K3, an integral homology Enriques surface or a rational homology torus bundle over torus.
\end{theorem}
The following table lists possible rational homological invariants of symplectic CY surfaces \cite{L4}:

$$\begin{tabular}{|c|c|c|c|c|c|}
\hline
$b_1$ & $b_2$ & $b^+$ & $\chi$ & $\sigma$ & known manifolds\\
\hline
0 & 22 & 3 & 24 & -16 & K3\\
\hline
0 & 10 & 1 & 12 & -8 & Enriques surface\\
\hline
4 & 6 & 3 & 0 & 0 & 4-torus\\
\hline
3 & 4 & 2 & 0 & 0 & $T^2-$bundles over $T^2$\\
\hline
2 & 2 & 1 & 0 & 0 & $T^2-$bundles over $T^2$\\
\hline
\end{tabular}
$$

\begin{proof}[Proof of Theorem \ref{int3}]
It follows from Propositions \ref{compare}, \ref{cy1}, Theorem \ref{cy2} and the table above.
\end{proof}

It is also speculated that a symplectic CY surface is diffeomorphic to the K3 surface, the Enriques surface or a torus bundle over torus. Thus we make the following

\begin{conj}\label{00} 
If $X$ is a K3 surface, or an Enriques surface, then under a
Luttinger surgery along any embedded Lagrangian
torus, $\tilde X$ is diffeomorphic to X.
\end{conj}

As for torus bundles over torus, we have

\begin{conj}\label{torus}
Any smooth oriented torus bundle $X$ over torus possesses a symplectic structure $\omega$ such that $(X,\omega)$ can be obtained by applying Luttinger surgeries to $(T^4,\omega_{std})$.
\end{conj}

In the list of torus bundles over torus in \cite{Ge}, any manifold in classes
(a), (b) and (d) has a Lagrangian bundle structure.
For any such manifold, it is not hard to verify Conjecture \ref{torus}
via Luttinger surgery along Lagrangian fibers.

\begin{remark}\label{44}
\begin{enumerate}
\item
Conjectures \ref{00} and \ref{torus} are clearly related to Question 2.6 in \cite{A}.

\item
There is another symplectic CY surgery in
dimension six, the symplectic conifold transition.
If $(M^6, \omega)$ with $K_\omega=0$ contains disjoint Lagrangian spheres $S_1, ...,
S_n$ with homology relations generated by
$\sum_{i=1}^n\lambda_i[S_i]=0$ with all $\lambda_i\ne 0$, Smith, Thomas, Yau (\cite{STY}) construct from
$(M^6, \omega)$ a new symplectic manifold $(M',\omega')$ with $K_{\omega'}=0$ and
smaller $b_3$.

\item
We notice that there is a parametrized Luttinger surgery in higher dimension and believe it
also should be a symplectic CY surgery. This will be discussed elsewhere.
\end{enumerate}
\end{remark}


\section{Topological preferred framing and Lagrangian framing}
In this section, we will introduce topological preferred framings and compare them with the Lagrangian framing for Lagrangian tori in $\kappa \leq 0$ symplectic 4-manifolds.

Suppose $X$ is a smooth 4-manifold and $L\subset X$ is an embedded 2-torus with trivial normal bundle.
Recall that a framing is a diffeomorphism $\varphi :U\rightarrow T^2\times D^2$ for a tubular neighborhood $U$ of $L$ such that $\varphi^{-1}(T^2\times 0)=L$ and a longitudinal curve of $\varphi$ is a lift $\gamma_{\varphi}$ of some simple closed curve $\gamma\subset L$ in $Z$. 
Let
$$H_{1,\varphi}:=<[\gamma_\varphi]|\gamma_\varphi :\hbox{ longitudinal curve of }\varphi>$$
$H_{1,\varphi}$ is a subgroup of $H_1(Z; \Z)$ and it induces a decomposition of $H_1(Z; \Z)$:
$$H_1(Z; \Z)=<[\mu]>\oplus H_{1,\varphi}.$$
Conversely, any rank 2 subgroup $V$ of $H_1(Z;\Z)$ such that $[\mu]$ and $V$ generate $H_1(Z;\Z)$ corresponds to a framing of $L$.

In \cite{Lu}, Luttinger introduced a version of topological preferred framings of Lagrangian tori in $\R^4$.
It requires that $H_{1,\varphi}$ is in the kernel of $i_1^\Z$.
On the other hand, Fintushel and Stern (\cite{FS}) defined null-homologous framings for a null-homologous torus via  $i_2^\Z$  
(seemingly, under the assumption that $H_1(X;\Z)$ vanishes, though  not explicitly mentioned).

The following definition is essentially the same as in \cite{FS}, but without assuming  that $H_1(X;\Z)$ vanishes.

\begin{definition}\label{tprefer}
Suppose $L$ is null-homologous, i.e., $[L]=0$ in $H_2(X;\Z)$. 
A framing $\varphi$ is called a topological preferred framing if $[L_{\varphi}]\in \ker i_2^\Z$.
Here, $L_{\varphi}\subset Z$ is a longitudinal torus of $\varphi$ given by $\varphi^{-1}(T^2\times z)$, $z\in \partial D^2$.
\end{definition}

There is the following  generalization when $[L]$ is a torsion class in $H_2(X;\Z)$. 

\begin{definition}\label{rtprefer}
Assume $[L]$ is a torsion class in $H_2(X;\Z)$. A framing $\varphi$ is called a rational topological preferred framing if $[L_{\varphi}]\in \ker i_2^\Q$.
\end{definition}

When $L$ is null-homologous, it is clear that a topological preferred framing is also a rational topological preferred framing.


\subsection{Comparing $\ker i_1$ and $\ker i_2$}
In order to compare various preferred framings and the Lagrangian framing, we need to investigate the relation of the maps $i_1$ and $i_2$ given by \eqref{i1} and \eqref{i2}.
Let $Y$ be a smooth oriented 4-manifold with boundary $Z=T^3$.

\begin{lemma}\label{kernel}
The maps $i_{1}$ and $i_{2}$ satisfy the following properties
\begin{enumerate}
\item
$r(\ker i_1)+r(\ker i_2)=3$.
\item
With the pairing
$$H_1(Z)\times H_2(Z)\rightarrow H_0(Z)\cong \Q,$$ 
given by the cap product,
$\ker i_2$ and $\ker i_1$ annihilate each other:
$$\ker i_2=\mathbf{ann}(\ker i_1):=\{c\in H_2(Z)|a\cdot c =0\in H_0(Z)\hbox{ for any }a\in \ker i_1\}$$
and $\ker i_1=\mathbf{ann}(\ker i_2)$.
\item $r(\ker i_1)> 0$.

\end{enumerate}
\end{lemma}
\begin{proof}
\begin{enumerate}
\item Consider the exact sequence
$$\cdots  \stackrel{\partial_2}{\longrightarrow} H_2(Z) \stackrel{i_{2}}{\longrightarrow}
 H_2(Y)  \stackrel{\delta_2}{\longrightarrow}
 H_2(Y, Z) \stackrel{\partial_{1}}{\longrightarrow}
 H_1(Z) \stackrel{i_{1}}{\longrightarrow}
 H_1(Y) \stackrel{\delta_1}{\longrightarrow} \cdots$$
 It induces a short exact sequence
  \begin{equation}\label{se}
0  \longrightarrow
 H_2(Y)/\hbox{Im} i_{2}  \stackrel{\nu_2}{\longrightarrow}
 H_2(Y, Z) \stackrel{\partial_{1}}{\longrightarrow} \hbox{Im}\partial_{1}=\ker i_{1}\longrightarrow 0
 \end{equation}
 By Lefschetz duality and universal coefficient theorem,
 $$H_2(Y, Z)\cong H^2(Y) \cong H_2(Y)$$ 
and $r(H_2(Y,Z))=r(H_2(Y))$. So \eqref{se} implies that
$$r(\ker i_1)=r(\hbox{Im}i_2)=r(H_2(Z))-r(\ker i_2)$$
which is (1).

\item
Consider the dual pairing
$$\begin{array}{ccccc}
H_1(Y)& \times & H^1(Y) & \stackrel{}{\rightarrow} & H_0(Y)\\
\uparrow i_1& & \downarrow j&& \uparrow\cong\\
H_1(Z)& \times & H^1(Z) & \stackrel{}{\rightarrow} & H_0(Z)\\
\end{array}$$

Because the maps $i_1$ and $j$ are induced by embedding and restriction, this pairing is natural, i.e., for $a\in H_1(Z)$ and $\alpha\in H^1(Y)$,
$$<i_1(a), \alpha>=<a, j(\alpha)>$$

There is an isomorphism of long exact sequences induced naturally by Lefschetz and Poincar\'e dualities:
$$\begin{array}{rcccccccccl}
\cdots \stackrel{}{\longrightarrow}
& H^1(Y, Z) & \stackrel{}{\longrightarrow}
& H^1(Y) & \stackrel{}{\longrightarrow}
& H^1(Z) & \stackrel{j}{\longrightarrow}
& H^2(Y, Z) & \stackrel{}{\longrightarrow} \cdots\\
 & \downarrow \cap [Y] & & \downarrow \cap [Y] & & \downarrow \cap [Z] & & \downarrow \cap [Y] &\\
\cdots \stackrel{i_{3}}{\longrightarrow}
& H_3(Y) & \stackrel{\delta_3}{\longrightarrow}
& H_3(Y, Z) & \stackrel{\partial_{2}}{\longrightarrow}
& H_2(Z) & \stackrel{i_{2}}{\longrightarrow}
& H_2(Y) & \stackrel{\delta_2}{\longrightarrow} \cdots
\end{array}$$

Using this diagram, the dual pairing induces the intersection pairing:
$$\begin{array}{ccccc}
&& \downarrow\\
&& H_3(Y) \\
&& \downarrow\\
H_1(Y)& \times & H_3(Y,Z) & \stackrel{}{\rightarrow} & H_0(Y)\\
\uparrow i_1& & \quad \downarrow \partial =\partial_2&& \uparrow\cong\\
H_1(Z)& \times & H_2(Z) & \stackrel{}{\rightarrow} & H_0(Z)\\
&& \downarrow i_2\\
&& H_2(Y) \\
&& \downarrow\\
\end{array}$$

Given $z_2\in \ker i_2$, there exists $z_3\in  H_3(Y,Z)$ such that $\partial z_3=z_2$.
Let $\beta$ be the Lefschetz dual of $z_3$ in $H^1(Y)$.
For any $z_1\in \ker i_1$,
$$z_1\cdot z_2=z_1\cdot \partial z_3=i_1(z_1)\cdot z_3=<i_1(z_1), \beta>=0$$
It shows that $\ker i_2\subset \mathbf{ann}(\ker i_1)$. 
From part (1),
$$r(\mathbf{ann}(\ker i_1))=r(H_2(Z))-r(\ker i_1)=r(\ker i_2).$$  
So $\ker i_2=\mathbf{ann}(\ker i_1)$.
Similar argument shows that $\ker i_1=\mathbf{ann}(\ker i_2)$.

\item
Suppose $r(\ker i_1)=0$, then $r(\ker i_2)=3$ and $i_2$ is the zero map by part (1).
Let $T_1, T_2$ be two nonisotopic embedded tori in $Z$ intersecting in a curve $C$ transversely.
Since $Z=T^3$, $[T_1]\cap [T_2]=[C]\neq 0$ in $H_1(Z)$.
Meanwhile, each $T_i$ bounds a 3-manifold $W_i$ in $Y$.
$W_1\cap W_2$ is a 2-cycle
whose boundary is $C$ and $[C]$ is in the kernel of $i_1$,
which contradicts the assumption that $i_1$ is injective.
\end{enumerate}
\end{proof}

Here  is a geometric interpretation of this lemma.
Assume $z_2$ is an integral class of $\ker i_2$ and $C$ is a closed curve in $Z$ such that $[C]\cdot z_2\neq 0$ in $Z$.
There exists a relative 3-cycle $W$ in $(Y, Z)$ such that $[\partial W]=z_2$.
In particular, we can assume that $W$ intersects $Z$ transversely and $\partial W$ intersects $C$ transversely at $a_1, \cdots , a_p$ and $b_1, \cdots,b_n$ in $Z$ with positive and negative intersections respectively.
Furthermore, we can give a collar structure $V\cong Z\times [0,\epsilon )$ near $Z$ and assume $W\cap V=\partial W\times [0,\epsilon)$.
If we push $C$ to  $C'=C\times\frac{\epsilon}{2}$ in the interior of $Y$,
then $C'$ and $W$ intersect transversely at $a_1\times\frac{\epsilon}{2}, \cdots , a_p\times\frac{\epsilon}{2}$ and $b_1\times\frac{\epsilon}{2}, \cdots, b_n\times\frac{\epsilon}{2}$ with positive and negative intersections respectively. Hence $[C']=i_1([C])$ and
$$[C']\cdot [W]=[C]\cdot [\partial W]=p-n=[C]\cdot z_2\neq 0$$
So $[C]$ can not be in $\ker i_1$.

\begin{remark}
\begin{enumerate}
\item
Part (1) of Lemma \ref{kernel} is still true in arbitrary dimension.
If $Y$ is a $(n+1)$-dimensional manifold with connected boundary $Z$ and $i_k: H_k(Z)\rightarrow H_k(Y)$ denotes the homomorphism induced by the inclusion $Z\rightarrow Y$, then
$$r(\ker i_{k-1})+r(\ker i_k)=r(H_k(Z))$$
for $2\leq k\leq n-1$.
\item Part (3) of Lemma \ref{kernel} is pointed out by Robert Gompf.
\end{enumerate}
\end{remark}

In the following, we give examples to illustrate Lemma \ref{kernel} according to $r(\ker i_1)$.

\begin{enumerate}
\item
Let $K_0$ be the trivial knot in $S^3$ and $X=S^1\times S^3$.
The complement of the torus $L=S^1\times K_0$ is
$$Y=S^1\times (S^3-K_0)\cong S^1\times (S^1\times D^2)$$
If $t, m$ denote the isotopy classes of these two $S^1$ and $l=\partial D^2$,
then $H_1(Z)=<[t], [m], [l]>$ and $\ker i_1$ has rank 1 which is  generated by $[l]$.
On the other hand, $\ker i_2$ is generated by $[l\times t], [l\times m]$ and has rank 2.

In general, if $K$ is any knot in $S^3$ and $S$ is a Seifert surface with boundary $K$,
we can define $t$ and $m$ as above and choose $l$ as the push-off of $K$ in $S$.
Then $Y$ and $T^2\times D^2$ have isomorphic homology groups and $\ker i_1$ is still generated by $l$, which bounds the surface $S$.
Similarly, $\ker i_2$ has rank 2 and is generated by $[l\times t], [l\times m]$.
They bound $S^1\times S$ and $\{pt\}\times (S^3-K)$ respectively.
In \cite{FS0},  Fintushel and Stern use these manifolds as building blocks to define knot surgery in 4-manifolds.

In the next example, the results of Lemma \ref{kernel} are not obvious.
Let $\pi:X\rightarrow \Sigma_g$ be a ruled surface and the loop $\gamma\subset X$ be a lift of a loop in $\Sigma_g$.
We can construct a torus $L$ in $X$ as the product of $\gamma$ and some circle $b$ in the fiber.
If $\mu\subset Z$ is a meridian of $L$ and $\pi(\gamma)$ is nontrivial in $\pi_1(\Sigma_g)$, it is easy to show that $\ker i_1$ is generated by a push-off of $b$.
But $\ker i_2$ is not obvious even when $X=S^2\times \Sigma_g$ is the trivial bundle.
By Lemma \ref{kernel}, we know that $\ker i_2$ has rank 2 and is generated by $[\mu\times b]$ and $[\mu\times \gamma]$.

\item
Let $L=a\times b$ be the Clifford torus embedded in the rational manifold $X= \C\PP^2$.
The group $H_1(Z)$ is generated by $[a], [b]$ and the meridian $[\mu]$.
It is easy to show that $\ker i_1 = <[a], [b]>$ has rank 2 and $\ker i_2 = <[a\times b]>$.

In general, if $X$ is simply connected and $L\subset X$ is a torus with trivial normal bundle,
then $r(\ker i_1)=2$ if and only if $[L]=0$ in $H_2(X)$.

\item 
If $r(\ker i_1)=3$, it follows from Lemma \ref{meridian} that such surgery will not change the homology for any torus surgery.
\end{enumerate}

There are similar results for Lemma \ref{kernel} over $\Z$ if we consider $r(\cdot)$ as the rank of abelian groups. In particular, the following lemma is the analogue of \ref{kernel}(2).

\begin{lemma}\label{zker}
With the pairing
 $$H_1(Z;\Z)\times H_2(Z;\Z)\rightarrow H_0(Z;\Z)\cong \Z,$$
given by the cap product, $\ker i_2^\Z$ annihilates $\ker i_1^\Z$:
$$\ker i_2^\Z\subset \mathbf{ann}_\Z(\ker i_1^\Z)=\{c\in H_2(Z;\Z)|a\cdot c=0\in H_0(Z;\Z)\hbox{ for any }a\in \ker i_1^\Z\}.$$
\end{lemma}
\begin{proof}
By Lemma \ref{kernel}(2), $\ker i_1^\Q=\mathbf{ann}(\ker i_2^\Q)$.
If $H_1(Z;\Z)$ is considered as the integral elements of $H_1(Z;\Q)$, then $\ker i_1^\Q =\ker i_1^\Z\otimes \Q$ and $\ker i_1^\Z\subset \ker i_1^\Q$.
So $\ker i_1^\Z \subset \mathbf{ann}(\ker i_2^\Q)=\mathbf{ann}(\ker i_2^\Z)$.
\end{proof}





\subsection{Preferred framings via $\ker i_1$}\label{PF}
Now we characterize topological preferred framings via $i_1$.
We first consider the rational ones.
\begin{prop}\label{rexist}
Assume $[L]=0$ in $H_2(X;\Q)$ and $\varphi$ is a framing of $L$.
Then $\varphi$ is a rational topological preferred framing if and only if $\ker i_1^\Q\subset H_{1,\varphi}\otimes \Q$.
\end{prop}
\begin{proof}
\begin{eqnarray*}
&&\varphi\hbox{  is a rational topological preferred framing}\\
&\Leftrightarrow &  [L_{\varphi}]\in \ker i_2^\Q\\
&\Leftrightarrow &  <[L_{\varphi}]>\subset \ker i_2^\Q\\
&\Leftrightarrow &  \mathbf{ann}(<[L_{\varphi}]>)\supset \mathbf{ann}(\ker i_2^\Q)\\
&\Leftrightarrow & \ker i_1^\Q\subset H_{1,\varphi}\otimes \Q. \ \hbox{(Lemma \ref{kernel})}
\end{eqnarray*}
\end{proof}

In the integral cases, we have

\begin{prop}\label{exist}
Suppose $L$ is null-homologous and $\varphi$ is a framing of $L$. Then
\begin{enumerate}
\item
$L$ has topological preferred framings, and
\item
$\ker i_1^\Z\subset H_{1,\varphi}$ if $\varphi$ is a topological preferred framing.
\end{enumerate}
\end{prop}

\begin{proof}
\begin{enumerate}
\item
Since $[L]=0$ in $H_2(X;\Z)$, there exists a 3-chain $W$ such that $\partial W=L$. We can assume that $W$ intersects $Z$ transversely.
In fact, we can choose a framing $\varphi :U\rightarrow T^2\times D^2$ such that $W\cap U=\varphi^{-1}(T^2\times S_x)$, where $S_x=\{(x, 0)\in D^2|x\geq 0\}$.
Then $W\cap Z=\varphi^{-1}(T^2\times (1, 0))$ is a longitudinal torus of $\varphi$ and $W\cap Y$ is a relative 3-cycle of $(Y, Z)$ with $\partial (W\cap Y)=W\cap Z$.
So $[W\cap Z]\in \ker i_2^\Z$ and $\varphi$ is a topological preferred framing.

\item
$[L_{\varphi}]\in \ker i_2^\Z$ implies that $\mathbf{ann}_\Z([L_{\varphi}])\supset \mathbf{ann}_\Z(\ker i_2^\Z)$.
It is easy to observe that $\mathbf{ann}_\Z(<[L_\varphi]>)=H_{1,\varphi}$.
By Lemma \ref{zker}, $$\ker i_1^\Z\subset \mathbf{ann}_\Z(\ker i_2^\Z)\subset \mathbf{ann}_\Z([L_{\varphi}])=H_{1,\varphi}.$$
\end{enumerate}
\end{proof}

If $[L]$ is torsion in $X$, Proposition \ref{exist}(1) may fail in two situations.
First, there may exist $a\in H_1(Z; \Z)$ such that $a\notin \ker i_1^\Z$ but $ka\in \ker i_1^\Z$ for some nonzero integer $k$.
So we can only define rational topological preferred framings.
Second, $[\mu]$ and $\ker i_1^\Z$ might not generate the group $H_1(Z; \Z)$.
In this case, rational topological preferred framings also do not exist.

\begin{remark}\label{lambda}
\begin{enumerate}
\item
In knot theory, the notion of preferred framings is similar to that of Definition \ref{tprefer}.
Let $M$ be an integral homology 3-sphere and $K\subset M$ be a knot.
If $V$ is a tubular neighborhood of $K$, a diffeomorphism $h:S^1\times D^2\rightarrow V$ satisfying $h(S^1\times 0)=K$ is called a framing of $K$.
Furthermore, $h$ is called a preferred framing if $h(S^1\times a)$ is homologically trivial in $M-V$.
For any knot $K$ in $M$, preferred framings exist and are unique up to isotopy (\cite{Ro}).

\item 
It is easy to see from Proposition \ref{exist} that Luttinger's definition coincides with Definition \ref{tprefer} when $X=\R^4$.

\item
In \cite{FS}, an invariant $\lambda(L)$ is defined when $[L]=0$ and $L$ has a unique topological preferred framing $\varphi_0$.
Assume $\varphi_{Lag}$ is the Lagrangian preferred framing.
Then $\varphi_{Lag}=\varphi_0$ if and only if  $\lambda(L)=0$.
Otherwise, $\lambda(L)$ is the smallest positive integer $k$ such that $k[\mu]+[\gamma_\varphi]\in H_{1,\varphi_{Lag}}$ for some $[\gamma_\varphi]\in H_{1,\varphi_0}$.

\end{enumerate}
\end{remark}


\subsection{$(1,k)$-surgeries and topological preferred framings}

The following proposition relates rational topological preferred framings and $(1,k)$-surgeries.

\begin{prop}\label{r1k}
Suppose $X$ is a smooth 4-manifold and $L\subset X$ is a torus with trivial normal bundle such that $[L]=0$ in $H_2(X;\Q)$ and $\varphi$ is a framing of $L$. Let $\tilde{X}=X_{(L, \varphi,\gamma, 1, k)}$ be constructed from $X$ via $(1,k)$-surgery along $(L,\varphi,\gamma)$.
\begin{enumerate}
\item
If $\varphi$ is a rational topological preferred framing of $L$, then  $\tilde{X}$ satisfies
$$r(H_1(\tilde{X}))=r(H_1(X))$$
for any $\gamma$ and $k$.
\item
If $H_1(\tilde{X};\Z)\cong H_1(X;\Z)$ for any $\gamma$ and $k$, then $\varphi$ is a rational topological preferred framing of $L$.
\end{enumerate}
\end{prop}
\begin{proof}
\begin{enumerate}
\item
By Lemma \ref{meridian}(3), we have
$$r(H_1(X))=
 \left\{\begin{array}{ll} r(H_1(Y))-1 & \hbox{if } i_1[\mu]\neq 0 \\
r(H_1(Y)) & \hbox{if } i_1[\mu]=0.
\end{array}\right.$$
Lemma \ref{mul} implies that $i_1[\mu]\neq 0$ in $H_2(Y)$. 
Since $\varphi$ is a rational topological preferred framing, Proposition \ref{rexist} implies that 
$[\mu]+k[\gamma_\varphi]\notin \ker i_1$ for any integer $k$ and $[\gamma_\varphi]\in H_{1,\varphi}$.
So
$$r(H_1(\tilde{X}))=r(H_1(Y))-1=r(H_1(X))$$
if $\tilde{X}$ is given via $(1, k)$-surgery.

\item
We first prove that any $a\in \ker i_1^\Z$ lies in $H_{1,\varphi}$.
Let $a=s[\mu]+t[\gamma_\varphi]$ for some $s, t\in \Z$ and $\gamma\subset L$ (Recall that $\gamma_\varphi$ is a lift of $\gamma$).
For $\tilde{X}=X_{(L,\varphi,\gamma, 1, kt)}$, the meridian of $\tilde{L}$ in $\tilde{X}$ satisfies
\begin{equation}\label{mad}
[\tilde{\mu}]=[\mu]+kt[\gamma_\varphi]=[\mu]+k(a-s[\mu])=(1-sk)[\mu]+ka
\end{equation}
So

$\begin{array}{rclr}
H_1(\tilde{X};\Z) &=& H_1(Y;\Z)/<i_1^\Z((1-sk)[\mu]+ka)> &
$(by \eqref{xy})$\\
&=& H_1(Y;\Z)/<i_1^\Z((1-sk)[\mu])> & (a\in \ker i_1^\Z)\\
&\cong& H_1(Y;\Z)/<i_1^\Z[\mu]>
\end{array}$

Because $[\mu]$ is essential in $Y$ and $k$ is arbitrary, $s$ should be zero and $a\in H_{1,\varphi}$.
Otherwise, $H_1(X;\Z)$ has infinitely many torsion classes with different orders.

Tensoring with $\Q$, we have 
$$\ker i_1 =\ker i_1^\Z\otimes \Q \subset H_{1,\varphi}\otimes \Q$$
By Proposition \ref{rexist}, $\varphi$ is a rational topological preferred framing. 
\end{enumerate}
\end{proof}

For integral cases, we have

\begin{prop}\label{1k}
Suppose $H_1(X;\Z)$ has no torsion and $L$ is null-homologous. 
Then a framing $\varphi$ of $L$ is a topological preferred framing if and only if
 $H_1(\tilde{X};\Z)\cong H_1(X;\Z)$ for any $\tilde{X}=X_{(L, \varphi,\gamma, 1, k)}$ obtained from $X$ via $(1,k)$-surgery.
\end{prop}
\begin{proof}
Assume $\varphi$ is a topological preferred framing.
Consider the 3-chain $W$ given in the proof of Proposition \ref{exist}.
It is clear that a meridian $\mu$ intersects $W$ at one point.
So $i_1^\Z[\mu]\cdot [W]=\pm 1$ and $i_1^\Z[\mu]$ is a primitive class.
Similarly, the simple closed curve $\tilde{\mu}$ has class $[\mu]+k[\gamma_\varphi]$ and is homotopic to a curve intersecting $W$ at one point.
Hence $i_1^\Z[\tilde{\mu}]$ is also a primitive class for any $\gamma,k$.
Since $H_1(Y;\Z)$ is a free abelian group, we have $H_1(Y,\Z)/<i_1^\Z[\mu]>\cong H_1(Y,\Z)/<i_1^\Z[\tilde{\mu}]>$.
So $H_1(X,\Z)\cong H_1(\tilde{X};\Z)$ by Lemma \ref{meridian}(3).

Conversely, if $H_1(X,\Z)\cong H_1(\tilde{X};\Z)$ for any $\tilde{X}$, the proof of Proposition \ref{r1k}(2) shows that $\varphi$ is a rational topological preferred framing.
The assumption that $H_2(X;\Z)$ has no torsion implies that $\varphi$ is actually a topological preferred framing.
\end{proof}

The knot surgery in \cite{FS0} is an example of $(1,k)$-surgeries.
Suppose $X_0=S^3\times S^1$, $K$ is a knot in $S^3$ and $L=K\times S^1$.
Let $\mu$ be the meridian of $L$, $a$  the longitude of the preferred framing of $K$ and $b=S^1$.
Then $\ker i_1^\Z=<[a]>$ and $\ker i_2^\Z=<[\mu\times a], [a\times b]>$.
Consider the framings $\varphi_p$ of $L$ where $H_{1,\varphi_p}$ is generated by $p[\mu]+[a]$ and $[b]$.
It is clear that $\varphi_p$ is a topological preferred framing of $L$ exactly when $p=0$.

If $K$ is the trivial knot and $[\gamma]=p[\mu]+[a]$, the resulting manifold of $(1, k)$-surgery is
$$S^3_{(L, \varphi_p, \gamma, 1, k)}\cong L(1+pk, k)\times S^1$$
and $H_1(S^3_{(L, \varphi_p, \gamma, 1, k)};\Z)\cong \Z_{1+kp}\oplus \Z$.
If $|p|>1$, $H_1(S^3_{(L, \varphi_p, \gamma, 1, k)};\Z)$ has rank $1$ for any $p ,k$, but the torsion subgroups vary.
Propositions \ref{1k} and \ref{r1k} show that $\varphi_p$ is not a rational topological preferred framing.
Actually, such framings do not exist for $L$.

\subsection{Constraints for Lagrangian framings}

Here we provide topological constraints on the isotopy classes of Lagrangian tori in many symplectic manifolds with non-positive Kodaira dimension.
In particular, they imply that the invariant $\lambda(L)$ of Fintushel and Stern (Remark \ref{lambda}(3)) is zero if the manifold has non-positive Kodaira dimension and vanishing integral $H_1$.
Recall that the Lagrangian framing for Lagrangian tori 
is defined in section \ref{construction}.

\begin{prop}\label{prefer}
Suppose $L$ is a Lagrangian torus in $(X,\omega)$  and
any Luttinger surgery along $L$ preserves the integral homology. 

\begin{enumerate}
\item If $H_2(X;\Z)$ is torsion free and $L$ is null-homologous,  then the Lagrangian framing of $L$ is a topological preferred framing.

\item If $[L]$ is torsion  then the Lagrangian framing of $L$ is a rational topological preferred framing.
\end{enumerate}

\end{prop}

\begin{proof}
The result follows directly from  Propositions \ref{1k} and \ref{r1k}.
\end{proof}

In particular, we have

\begin{cor}\label{ltp1} 
If $\kappa(X)=-\infty$ and $L$ is a Lagrangian torus in $X$, then the Lagrangian
framing of $L$ is  a topological preferred framing.
\end{cor}

\begin{proof} Since $b^+(X) = 1$ and $H_2(X;\Z)$ has no torsion, $L$ is null-homologous.
For any $\tilde X$ given from $X$ by applying Luttinger surgery along $L$, $\tilde X$
is diffeomorphic to $X$ by Theorem \ref{int2}.
Now the claim follows from Proposition
\ref{prefer}.
\end{proof}

In the case that  $X$ is a symplectic CY surface, it is convenient to introduce 

\begin{definition}\label{ce}
An embedded 2-torus $L$ of a 4-manifold $X$ is called essential if $[L]\neq 0$ in $H_2(X)$. 
Moreover, $L$ is called completely essential if $r(\ker i_1^\Q)=3$.
\end{definition}

\begin{prop}\label{ltp2}
If $\kappa(X)=0$ and $X$ is an integral homology
K3, then any Lagrangian torus $L$ satisfies one of
the following conditions:

\begin{enumerate}
\item $L$ is null-homologous  and the Lagrangian framing is a topological preferred framing.

\item $L$ is completely essential.
\end{enumerate}
\end{prop}
\begin{proof}
When $L$ is null-homologous, the claim follows from Theorem \ref{int3} 
and
Proposition \ref{prefer}.

When $L$ is essential, $[\mu]\in\ker i_1^\Q$ by Lemma \ref{mul}. 
If $r(\ker i_1^\Q)\neq 3$, there exists $\gamma\subset L$ such that $i_1^\Q ([\gamma_\varphi ])\neq 0$.
In the manifold $\tilde{X}=X(L, \gamma, 1)$, the class $[\tilde{\mu}]=[\mu]+[\gamma_\varphi]$ is nonzero in $H_1(Y)$.
By Proposition \ref{compare}, $b_1(X)\neq b_1(\tilde{X})$, which contradicts Theorem \ref{int3}.
So $r(\ker i_1^\Q)= 3$ and $L$ is completely essential.
\end{proof}

Similarly we have 
\begin{prop}\label{ltp3}
If $\kappa(X)=0$ and $X$ is an integral homology
 Enriques surface, then any Lagrangian torus $L$ satisfies one of
the following conditions:

\begin{enumerate}
\item $L$ is null-homologous  and the Lagrangian framing is a topological preferred framing.
\item $[L]$ is  torsion  and the Lagrangian framing is a rational topological preferred framing.

\end{enumerate}
\end{prop}

\begin{proof}[Proof of Theorem \ref{int4}]
The first statement follows from Theorem \ref{cy2}, Corollary \ref{ltp1}, Propositions \ref{ltp2}(1) and \ref{ltp3}(1).
The last statement on $\lambda (L)$ follows from Remark \ref{lambda}(3).
\end{proof}



\end{document}